\documentclass{article}
\usepackage{amssymb}
\usepackage{amsthm}
\font\bigmath=cmsy10 scaled \magstep 3
\newcommand{\bigtimes}{\hbox{\bigmath \char'2}}
\newcommand{\emp}{\emptyset}
\newcommand{\ben}{\mathbb N}
\newcommand{\ber}{\mathbb R}
\newcommand{\beq}{\mathbb Q}
\newcommand{\bez}{\mathbb Z}
\newcommand{\nhat}[1]{\{1,2,\ldots,#1\}}
\newcommand{\ohat}[1]{\{0,1,\ldots,#1\}}
\newcommand{\mod}{\hbox{\rm mod }}
\newcommand{\ggamma}{\raise 2 pt \hbox{$\gamma$}}

\newtheorem{theorem}{Theorem}[section]
\newtheorem{corollary}[theorem]{Corollary}

\newtheorem{lemma}[theorem]{Lemma}
\theoremstyle{definition}
\newtheorem{definition}[theorem]{Definition}

\begin{document}

\title{Multiply partition regular matrices}

\author{Dennis Davenport \and Neil Hindman${}^1$ \and Imre Leader \and Dona Strauss}
\date{}

\maketitle

\addtocounter{footnote}{1}
\footnotetext{This author acknowledges support received from the National Science
Foundation (USA) under grant DMS-1160566.\hfill\break
{\it Keywords:} matrix, image partition regular, kernel partition regular, 
columns condition, Rado's Theorem\hfill\break
2010 MSC: 05D10}

\begin{abstract}

Let $A$ be a finite matrix with rational entries.  We say that $A$ is {\it doubly image partition
regular\/} if whenever the set $\ben$ of positive integers is finitely coloured,
there exists $\vec x$ such that the entries of $A\vec x$ are all the same colour (or {\it monochromatic\/})
and also, the entries of $\vec x$ are monochromatic. Which matrices are
doubly image partition regular? 

More generally, we say that a pair of matrices
$(A,B)$, where $A$ and $B$ have the same number of rows, is  
{\it doubly 
kernel partition regular\/}
if whenever $\ben$ is finitely coloured, there exist vectors $\vec x$ and
$\vec y$, each monochromatic, such that $A \vec x + B \vec y = 0$. There is an
obvious sufficient condition for the pair $(A,B)$ to be doubly kernel partition
regular, namely that 
there exists a positive rational $c$ such that the matrix
$M=(\begin{array}{ccccc}A&cB\end{array})$ is kernel partition
regular.  (That is, whenever $\ben$ is finitely coloured, there exists monochromatic $\vec x$
such that $M \vec x=\vec 0$.) Our aim in this paper is to show that this
sufficient condition is also necessary. 
As a consequence we have that a matrix $A$ is doubly image partition regular 
if and only if there is a positive rational $c$
such that the matrix $(\begin{array}{lr}A&cI\end{array})$ is kernel partition regular, where 
$I$ is the identity matrix of the appropriate size. 

We also prove extensions
to the case of several matrices.
\end{abstract}

\section{Introduction}

Let $u,v\in\ben$ and let $A$ be a $u\times v$ matrix with entries from $\beq$.  Then
$A$ is said to be {\it kernel partition regular\/} (abbreviated KPR) if and only if whenever 
$r\in\ben$ and $\varphi:\ben\to\nhat{r}$, there exists some $\vec x\in\ben^v$
such that $\varphi$ is constant on the entries of $\vec x$ and $A\vec x=\vec 0$.  In
the standard ``chromatic'' terminology, $\varphi$ is said to be an {\it $r$-colouring\/} of
$\ben$ and $\vec x$ is said to be {\it monochromatic\/}.  If $r$ is not specified, one
may say simply that $\ben$ is finitely coloured by $\varphi$.  The 
question of which matrices are KPR was solved in 1933 by Richard Rado \cite{R}.

\begin{definition}\label{defcol} Let $u,v\in\ben$ and let $A$ be a $u\times v$ matrix with entries from $\beq$
and let $\vec c_1,\vec c_2,\ldots,\vec c_v$ be the columns of $\vec A$.  
Then $A$ satisfies the {\it columns condition\/} if and only if there exists $m\in\nhat{v}$
and a partition $\{I_1,I_2,\ldots,I_m\}$ of $\nhat{v}$ such that
\begin{itemize}
\item[(1)] $\sum_{i\in I_1}\vec c_i=\vec 0$ and
\item[(2)] for each $t\in\{2,3,\ldots,m\}$, if any, 
$\sum_{i\in I_t}\vec c_i$ is a linear combination of $\{\vec c_i:i\in\bigcup_{j=1}^{t-1}I_j\}$.
\end{itemize}
\end{definition}

\begin{theorem}[Rado's Theorem]\label{Rado}
 Let $u,v\in\ben$ and let $A$ be a $u\times v$ matrix with entries from $\beq$.  Then
$A$ is kernel partition regular if and only if $A$ satisfies the columns condition. \end{theorem}

For example the fact that the matrix $(\begin{array}{ccc}1&1&-1\end{array})$ satisfies the 
columns condition (with $I_1=\{1,3\}$ and $I_2=\{2\}$) shows that whenever $\ben$ is finitely
coloured, there exist some monochromatic $x_1$, $x_2$, and $x_3$ with $x_1+x_2=x_3$, which
is Schur's Theorem \cite{S}.  

As another example, the length $4$ version of van der Waerden's Theorem \cite{W} says
that whenever $\ben$ is finitely coloured, there exist $a,d\in\ben$ such that
$\{a,a+d,a+2d,a+3d\}$ is monochromatic.  Letting $x_1=a$, $x_2=a+d$, $x_3=a+2d$,
$x_4=a+3d$, and $x_5=d$, the fact that the matrix
$$\left(\begin{array}{ccccc}
-1&1&0&0&-1\\
0&-1&1&0&-1\\
0&0&-1&1&-1\end{array}\right)$$
satisfies the columns condition (with $I_1=\{1,2,3,4\}$ and $I_2=\{5\}$)
shows that one can get monochromatic $\vec x$ with
$x_2-x_1=x_3-x_2=x_4-x_3=x_5$.

We remark that the above two examples were already known when Rado's 
Theorem was
proved.  The importance of Rado's Theorem is that it reduces the question of
whether or not a given matrix is kernel partition regular to a finite
computation. For example, the fact that the matrix
$$\left(\begin{array}{ccccccc}
1&1&0&-1&0&0&0\\
1&0&1&0&-1&0&0\\
0&1&1&0&0&-1&0\\
1&1&1&0&0&0&-1\end{array}\right)$$
satisfies the columns condition (with $I_1=\{1,4,5,7\}$, $I_2=\{2,6\}$, and
$I_3=\{3\}$) established the previously unknown extension of Schur's Theorem
that whenever $\ben$ is finitely coloured, there must exist
$x_1$, $x_2$, and $x_3$ with $\{x_1,x_2,x_3,x_1+x_2,x_1+x_3,x_2+x_3,x_1+x_2+x_3\}$
monochromatic.  (And it is easy to similarly establish extensions
for any finite number of terms.)

We now turn to the other key notion of partition regularity.
 
\begin{definition}\label{defipr} Let $u,v\in\ben$ and let
$A$ be a $u\times v$ matrix with rational entries.  Then $A$ is
{\it image partition regular\/} (abbreviated IPR) if and only if
whenever $\ben$ is finitely coloured, there exists $\vec x\in\ben^v$
such that the entries of $A\vec x$ are monochromatic.\end{definition}

Notice that the applications of Rado's Theorem cited above are very naturally
stated in terms of image partition regular matrices.  Specifically, Schur's 
Theorem, the length 4 version of van der Waerden's Theorem, and
the three term extension of Schur's Theorem are the assertions that the
following three matrices are image partition regular.
$$\begin{array}{ccc}
\left(\begin{array}{cc}
1&0\\
0&1\\
1&1\end{array}\right)
&\left(\begin{array}{cc}
1&0\\
1&1\\
1&2\\
1&3\end{array}\right)
&\left(\begin{array}{ccc}
1&0&0\\
0&1&0\\
1&1&0\\
0&0&1\\
1&0&1\\
0&1&1\\
1&1&1\end{array}\right)\end{array}$$

In view of the fact that many problems are very naturally stated as
questions about image partition regularity, it is surprising that IPR matrices
were not characterized until 1993 \cite{HL}.  Among the characterizations
obtained then were two that used the columns condition, and were therefore
computable.  Several other characterizations have been obtained since 
then. (See \cite[Theorem 2.10]{HLS} and \cite[Theorem 15.24]{HS}.)

It is very natural to ask the following question about image partition regular matrices. When can one insist that not only are the entries of $A\vec x$ 
monochromatic, but also that the entries of $\vec x$ are monochromatic, though
not necessarily of the same colour as the entries of $A\vec x$? It is
this question which motivates the current paper.

There are some finite matrices over $\beq$ which can be seen at a glance to have this property.
These are the matrices which have no zero rows and have the property that, for some positive
natural number $c$, the first non-zero entry in each row is equal to $c$. There are also very simple IPR matrices which do not have this property. The diagonal matrix $\left(\matrix{1&0\cr 0&2}\right)$ provides an example, as can be seen by  mapping
each positive integer to the starting position (mod 2) of its base 2 expansion.

\begin{definition}\label{defdipr} Let $u,v\in\ben$ and let $A$ be a 
$u\times v$ matrix with entries from $\beq$.  Then $A$ is {\it doubly
image partition regular\/} (abbreviated doubly IPR) if and only if
whenever $\ben$ is finitely coloured, there exists monochromatic
$\vec x\in\ben^v$ such that $A\vec x$ is monochromatic.\end{definition}

It is easy to see (or see below) a sufficient condition. Suppose that we can 
insist that, for some positive rational $c$, we actually have that all the
entries of $c \vec x$ are the same colour as the entries of $Ax$; then it
follows that $A$ is doubly IPR. One of our main aims in this paper is to
show that this sufficient condition is also necessary.

The following very simple fact relates the notion of doubly IPR to kernel
partition regularity. Given $n\in\ben$ we denote the $n\times n$
identity matrix by $I_n$. We have that if $A$ is a 
$u\times v$ matrix with entries from $\beq$ then $A$ is doubly
IPR if and only if whenever $\ben$ is finitely coloured, there
exist monochromatic $\vec x\in\ben^v$ and monochromatic $\vec y\in
\ben^u$ such that $A\vec x-I_u\vec y=\vec 0$.

This fact in turn motivates the following definition.  

\begin{definition}\label{defdkpr} 
Let $u,v,w \in\ben$.  
Let $A$ be a $u\times v$ matrix with entries from $\beq$ and let
$B$ be a $u\times w$ matrix with entries from $\beq$
Then $(A,B)$ is {\it doubly kernel partition regular\/} (abbreviated
doubly KPR) if and only if
whenever $\ben$ is finitely coloured, there exist 
monochromatic $\vec x$ and $\vec y$ such that $A \vec x + B \vec y=\vec 0$.
\end{definition}

So a matrix $A$ is doubly IPR if and only if the pair $(A,-I)$ is doubly KPR.
A key idea in our proof of the characterisation of doubly IPR is to shift our
attention from this `asymmetrical' case of $(A,-I)$ and to consider instead the
more general question of when $(A,B)$ is doubly KPR. Again, it turns out
(see below) that if there is a positive rational $c$ such that the matrix
$(\begin{array}{ccccc}A&cB\end{array})$ is KPR then $(A,B)$ is doubly KPR. We 
will show that this sufficient condition is in fact necessary.

More generally, we make the following definition.

\begin{definition}\label{defmkpr} 
Let $u,k,v_1,v_2,\ldots,v_k\in\ben$ with $k\geq 2$.  
For $t\in\nhat{k}$, let $A_t$ be a $u\times v_t$ matrix with entries from $\beq$.
Then $(A_1,A_2,\ldots,A_k)$ is {\it multiply kernel partition regular\/} (abbreviated
multiply KPR) if and only if
whenever $\ben$ is finitely coloured, there exist for each $t\in\nhat{k}$, 
monochromatic $\vec x_t$ such that $A_1\vec x_1+A_2\vec x_2+\ldots+A_k\vec x_k=\vec 0$.
\end{definition}

Section 2 of this paper consists of a proof of the fact that $(A_1,A_2,\ldots,A_k)$
is multiply KPR if and only if there exist positive rationals
$c_2,c_3,\ldots,c_k$ such that the 
matrix $(\begin{array}{ccccc}A_1&c_2A_2&c_3A_3&\ldots&c_kA_k\end{array})$ is KPR.
Section 3 consists of derivation of some consequences of this fact,
including the fact that the $u\times v$ matrix $A$ is doubly IPR
if and only if there is some positive rational $c$ such that
$\left(\begin{array}{c}cI_v\\A\end{array}\right)$ is IPR.  

We conclude this introduction with the following simple and well
known fact which we will use a couple of times.

\begin{lemma}\label{lemnN} Let $u,v,n\in\ben$ and let 
$A$ be a KPR $u\times v$ matrix with rational entries.
Then whenever $\ben$ is finitely coloured, there 
exists monochromatic $\vec x\in(n\ben)^v$ such that
$A\vec x=\vec 0$.\end{lemma}

\begin{proof} Let $\varphi$ be a finite colouring of $\ben$
and define a colouring $\psi$ of $\ben$ by $\psi(x)=\varphi(nx)$.
Pick $\vec y\in\ben^v$ which is monochromatic with respect to 
$\psi$ such that $A\vec y=\vec 0$.  Let $\vec x=n\vec y$.
Then $\vec x$ is monochromatic with respect to $\varphi$ and
$A\vec x=\vec 0$.\end{proof}

\section{Characterising multiply kernel partition regular matrices}

We begin with the easy half of the main theorem.
We write $\beq^{+}$ for the set of positive rationals.

\begin{lemma}\label{lemdkpr} Let $u,k,v_1,v_2,\ldots,v_k\in\ben$ with $k\geq 2$.  
For $t\in\nhat{k}$, let $A_t$ be a $u\times v_t$ matrix with entries from $\beq$.
If there exist $c_2,c_3,\ldots,c_k\in\beq^+$ such that
$(\begin{array}{ccccc}A_1&c_2A_2&c_3A_3&\ldots&c_kA_k\end{array})$ is KPR,
then $(A_1,A_2,\ldots,A_k)$ is multiply KPR.
\end{lemma}

\begin{proof} Assume that $c_2,c_3,\ldots,c_k\in\beq^+$ and
$(\begin{array}{ccccc}A_1&c_2A_2&c_3A_3&\ldots&c_kA_k\end{array})$ is KPR.
Let $\varphi$ be a finite colouring of $\ben$ and let $\psi$ be a finite colouring of
$\ben$ with the property that if $\psi(x)=\psi(y)$ then
\begin{itemize}
\item[(1)] $\varphi(x)=\varphi(y)$ and
\item[(2)] if $t\in\{2,3,\ldots,k\}$ and 
$c_tx$ and $c_ty$ are integers, then $\varphi(c_tx)=\varphi(c_ty)$.
\end{itemize}
For each $t\in\nhat{k}$, pick $m_t,n_t\in\ben$ such that $c_t=\frac{m_t}{n_t}$ 
and let $n=\prod_{t=2}^kn_t$.  Pick by Lemma \ref{lemnN}
$\vec z\in(n\ben)^{v_1+v_2+\ldots+v_k}$ which is monochromatic with respect to $\psi$ 
such that $(\begin{array}{ccccc}A_1&c_2A_2&c_3A_3&\ldots&c_kA_k\end{array})\vec z=\vec 0$.
For each $t\in\nhat{k}$, pick $\vec x_t\in(n \ben)^{v_t}$ such that
$$\vec z=\left(\begin{array}{c}\vec x_1\\ \vec x_2\\ \vdots\\ \vec x_k\end{array}\right)\,.$$
Then the entries of $\vec x_1$ are monochromatic with respect to $\varphi$ and
given $t\in\{2,3,\ldots,k\}$, since the entries of $\vec x_t$ are in $n\ben$, we have
that the entries of $c_t\vec x_t$ are monochromatic with respect to $\varphi$.
And $A_1\vec x_1+A_2c_2\vec x_2+\ldots+A_kc_k\vec x_k=\vec 0$.\end{proof}

The rest of this section will be devoted to a proof of the converse of 
Lemma \ref{lemdkpr}.  This proof is somewhat complicated, so we will
first present an informal description of the ideas of the proof for the
case $k=2$ (where we have a given doubly KPR pair $(A,B)$).

There are three key ingredients, two of which have appeared in other papers and
one of which is new.
\begin{itemize}
\item[(1)] The `start base p' colouring. This is used in \cite{HL}.
\item[(2)] Simple facts about linear spans and positive cones being closed 
sets.\hfill\break Again, these have been used in \cite{HL}.
\item[(3)] Looking at linear spans for `all parts of the partition at once'. 
This will be explained below, and it is the `new ingredient'.
\end{itemize}

Let us fix some notation. The columns of $A$ are $\vec a_1,\vec a_2,\ldots,\vec a_v$ and the 
columns of $B$ are $\vec b_1,\vec b_2,\ldots,\vec b_w$.

For a large positive integer $p$ (not necessarily prime), we colour the 
naturals by first two digits and start position $(\mod 2)$, all in the base $p$ 
expansion. So for example if $s$ is $67100200$ and $t$ is $3040567$ then $s$ gets 
colour $(67,1)$ and $t$ gets colour $(30,0)$. So we have  $2p(p-1)$ colours.

For this colouring, there are monochromatic vectors 
$\vec x=\left(\begin{array}{c}x_1\\ \vdots\\ x_v\end{array}\right)$ and 
$\vec y=\left(\begin{array}{c}y_1\\ \vdots\\ y_w\end{array}\right)$
with $A\vec x + B\vec y = \vec 0$. Say all 
the entries of $x$ start with the two digits $d$, where $d$ is between $1$ and $p$ 
(this is just for convenience of writing later on) -- so for example the 
above $s$ would have $d = 6 + \frac{7}{p}$ and the $t$ would have $d = 3$. And say all the 
entries of $y$ start with the two digits $e$.

We have an ordered partition of the index set of the columns of $A$ union the 
index set of the columns of $B$, according to which of the $x_i$ and $y_i$ start 
furthest to the left, which next furthest, and so on. We want to look at 
each set in the partition as its part in $A$ and its part in $B$. So we have a partition $D\cup D' \cup D''\cup\ldots$
of the columns of $A$, and a partition $E \cup E' \cup E''\cup\ldots$ of the columns 
of $B$, such that (and here note that we are allowed to have one of $D$ or $E$ empty 
but not both, and one of $D'$ or $E'$ empty but not both, etc.):
\begin{itemize}
\item[(1)] All the $x_i$ for $i\in D$ and all the $y_i$ for $i\in E$ start in the same 
place as each other
\item[(2)] All the $x_i$ for $i\in D'$ and all the $y_i$ for $i\in E'$ start in the same 
place as each other, and this place is to the right of the start-place for 
the $D,E$ terms by an even number of positions
\item[(3)] and so on.
\end{itemize}

For infinitely many $p$, this ordered partition (strictly speaking, this pair of ordered
partitions) is the same. And from now on we will always assume that our $p$ is
chosen from this infinite set.

We write $\vec s(D)$ for the sum of the columns of $A$ 
indexed by $D$, and also $\vec s(E)$ for the sum of the columns of $B$ indexed by $E$. 
And similarly for $\vec s(D')$ etc.

Consider the equation $A\vec x + B\vec y = \vec 0$. This says that the sum of all 
$x_ia_i$ plus the sum of all $y_ib_i$ is zero. If we consider 
dividing this by a suitable power of $p$, and using the fact that anything 
that starts to the right of the $x_i$ in $D$ actually starts at least two 
places to the right, we see that $d \vec s(D) + e \vec s(E) + \vec \delta = \vec 0$, where $\vec \delta$
denotes a certain sum of the columns of 
$(\begin{array}{cc}A &B\end{array})$, each with a coefficient that is at most $\frac{1}{p}$.

Now, normally one would proceed, by saying that this equation tells us that $\vec s(D)$ plus a multiple 
of $\vec s(E)$ equals $(-1/d) \vec 
\delta$, whence the vector $\vec s(D)$ is arbitrarily close 
to the positive cone on the vector $\vec s(E)$ (namely the set of all non-negative real 
multiples of the vector $\vec s(E)$). But positive cones are closed, hence in fact 
$\vec s(D)$ is a non-positive multiple of $S(E)$. And this 
would give us a first sum of columns of $(\begin{array}{cc}A &cB\end{array})$ that is zero.

However, instead of that, we will stick with that equation, for each 
fixed $p$, which is $d \vec s(D) + e \vec s(E) + \vec \delta = \vec 0$. Call this equation (1).

Now let us consider $\sum x_i a_i + \sum y_i b_i = 0$ 
when we divide by a different power of $p$, to focus on $D'$ and $E'$. We would 
get a term $d \vec s(D') + e \vec s(E')$, and a smaller contribution from columns not 
in $D,D',E,E'$ as well as the terms from $D'$ and $E'$ below the 
two most significant digits (with coefficients at most $\frac{1}{p}$), and also an unknown
contribution from the $x_i$ and $y_i$ that start to the left of where we are, 
namely the $x_i$ from $D$ and the $y_i$ from $E$.

So we have:

$d \vec s(D') + e \vec s(E') + \vec \delta' = \vec v$ for some $\vec v$ in the linear span of the columns 
of $D$ and $E$. Write this span as span$(D,E)$. 

In other words:

$d \vec s(D') + e \vec s(E') + \vec \delta'$   belongs to span$(D,E)$. This is equation (2).

Keep going. Next time we obtain:

$d \vec s(D'') + e \vec s(E'') + \vec \delta''$   belongs to span$(D,D',E,E')$. And so on. (We recall that this is for one fixed $p$. If we vary $p$, we will be changing
$d$ and $e$ and so on.)

We are now ready for the new ingredient. We do not wish to perform any 
limiting in equation (1) or (2). Rather, we want to look inside 
a product space. Let's say that the columns of our matrices live in $V$ 
(namely $\ber^u$). So as to keep the notation manageable, let us assume 
that our partitions are into 3 parts: so we have $D,D',D''$ (but there is no 
$D'''$) and same for $E,E',E''$. We now take the product of our equations. 
So, still for fixed $p$, in the space $V\times V\times V$ we have, combining (1),(2),(3):

$d \big(\vec s(D),\vec s(D'),\vec s(D'')\big)+ e\big(\vec s(E),\vec s(E'),\vec s(E'')\big)$ is very close to
the set $\{\vec 0\}\times \hbox{span}(D,E) \times \hbox{span}(D,D',E,E')$.

Note that this latter set, say $L$, is the linear span of a certain set of 
vectors (such as each vector $(\vec 0,\vec a_i,\vec 0)$ for $i \in D$). Which we can, if we wish, 
also view as the positive cone on (i.e. the non-negative linear 
combinations of) a certain finite set of vectors (namely the vectors we 
have just mentioned and their negatives).

Dividing by $d$, we see that $-\big(\vec s(D),\vec s(D'),\vec s(D'')\big)$ is 
arbitrarily close 
to the positive cone on $L\cup\big\{ \big(\vec s(E),\vec s(E'),\vec s(E'')\big)\big\}$. But positive cones (on 
finite sets of vectors) are closed sets, so, letting $p$ tend to infinity, we 
conclude that:
$-\big(\vec s(D),\vec s(D'),\vec s(D'')\big)$ is in
 the positive cone on $L\cup\big\{ \big(\vec s(E),\vec s(E'),\vec s(E'')\big)\big\}$.

In other words, there exists a nonnegative rational $c$ (switching from reals 
to rationals, which is fine as all coefficients are rationals in our matrices) 
such that:

$\big(\vec s(D),\vec s(D'),\vec s(D'')\big)+c\big(\vec s(E),\vec s(E'),\vec s(E'')\big)$
 belongs to $L$.

Case 1: $c$ is positive. In this case, looking at what $L$ is, we see that $(\begin{array}{cc}A&cB\end{array})$ satisfies the columns condition 
where the first block is $D \cup E$, then $D' \cup E'$, then $D''\cup E''$.

Case 2: $c=0$. This ought to be a trivial case, but in fact we do not know how
to eliminate it directly. Rather, let us
return to where we divided by 
$d$, and instead divide by $e$. In other words, we switch the roles of $A$ and $B$. 
We obtain that for some nonnegative rational $c'$ we have 
$\big(\vec s(E),\vec s(E'),\vec s(E'')\big)+c'\big(\vec s(D),\vec s(D'),\vec s(D'')\big)$
belongs to $L$.
 Again, if $c'$ is nonzero, we are done. So 
the only case left is when $c'=0$. This tells us that the point 
$\big(\vec s(E),\vec s(E'),\vec s(E'')\big)$ also belongs to $L$. But now it follows that for {\it any\/} 
positive rational $c$ at all (indeed, any nonzero $c$) the matrix
$(\begin{array}{cc}A&cB\end{array})$ satisfies the columns condition.
\medskip

Now we present a more formal version of the proof.

\begin{theorem}\label{thmdkpr}  Let $u,k,v_1,v_2,\ldots,v_k\in\ben$ with $k\geq 2$.  
For $t\in\nhat{k}$, let $A_t$ be a $u\times v_t$ matrix with entries from $\beq$.
Then $(A_1,A_2,\ldots,A_k)$ is multiply KPR if and only if there exist
$c_2,c_3,\ldots,c_k\in\beq^+$ such that
$$(\begin{array}{ccccc}A_1&c_2A_2&c_3A_3&\ldots&c_kA_k\end{array})$$ is KPR.\end{theorem}

\begin{proof} The sufficiency is Lemma \ref{lemdkpr}.  We shall prove the necessity.

 For $p\in\ben\setminus\{1\}$ define $g_p:\ben\to\omega$ by $g_p(x)=\max\{t\in\omega:
p^t\leq x\}$.  Define 
$\tau_p:\omega\times\ben\to\ohat{p-1}$ by $x=\sum_{j=0}^{g_p(x)}\tau_p(j,x)p^j$, letting
$\tau_p(j,x)=0$ if $j>g_p(x)$.  Define a finite colouring $\ggamma_p$ of $\ben$ so that
for $x,y\in\ben$, $\ggamma_p(x)=\ggamma_p(y)$ if and only if
\begin{itemize}
\item[(1)] $g_p(x)\equiv g_p(y)\ (\mod\ 2)$,
\item[(2)] $\tau_p(g_p(x),x)=\tau_p(g_p(y),y)$, and
\item[(3)] $\tau_p(g_p(x)-1,x)=\tau_p(g_p(y)-1,y)$.
\end{itemize}

For $p\in\ben\setminus\{1\}$ and $t\in\nhat{k}$, pick $\vec x_{t,p}\in\ben^{v_t}$ such
that $\vec x_{t,p}$ is monochromatic with respect to $\ggamma_p$ and
$$A_1\vec x_{1,p}+A_2\vec x_{2,p}+\ldots+A_k\vec x_{k,p}=\vec 0\,.$$
Pick $m_p\in\ben$,  $\mu_p(1)>\mu_p(2)>\ldots>\mu_p(m_p)$, and,
for each $t\in\nhat{k}$, pick pairwise disjoint sets $I_{t,p}(1),I_{t,p}(2),\ldots
I_{t,p}(m_p)$ such that \begin{itemize}
\item[(1)] for each $t\in\nhat{k}$, $\bigcup_{i=1}^{m_p}I_{t,p}(i)=\{t\}\times\nhat{v_t}$,
\item[(2)] for each $i\in\nhat{m_p}$, $\bigcup_{t=1}^k I_{t,p}(i)\neq\emp$, and
\item[(3)] for each $i\in\nhat{m_p}$ and each $(t,j)\in\bigcup_{s=1}^kI_{s,p}(i)$,
$g_p(x_{t,p,j})=\mu_p(i)$.\end{itemize}

Pick an infinite set $P\subseteq\ben$, $m\in\ben$, and for each $t\in\nhat{k}$ and 
each $i\in\nhat{m}$, $I_t(i)$, such that for each $p\in P$, $m_p=m$, and for each
$t\in\nhat{k}$ and each $i\in\nhat{m}$, $I_{t,p}(i)=I_t(i)$.

By reordering the columns of each $A_t$, and correspondingly reordering the entries
of each $\vec x_{t,p}$, we may presume that we have for each $t\in\nhat{k}$,
$0=\alpha_t(0)\leq \alpha_t(1)\leq\ldots\leq\alpha_t(m)=v_t$ such that
for each $i\in\nhat{m}$ and each $t\in\nhat{k}$, $I_t(i)=\{(t,j):\alpha_t(i-1)<j\leq\alpha_t(i)\}$.
Thus, if $p\in P$, $i\in\nhat{m}$, $t\in\nhat{k}$, and $\alpha_t(i-1)<j\leq\alpha_t(i)$,
then $g_p(x_{t,p,j})=\mu_p(i)$.  After the reordering, denote the columns of $A_t$ by
$\vec a_{t,1},\vec a_{t,2},\ldots,\vec a_{t,v_t}$.

For $i\in\nhat{m}$, let $J(i)=\bigcup_{t=1}^k I_t(i)$ and note that 
$\{J(1),J(2),\ldots,\break
J(k)\}$ is a partition of the indices of the columns of
$(\begin{array}{ccccc}A_1&A_2&A_3&\ldots&A_k\end{array})$.

For each $i\in\nhat{m}$ and each $t\in\nhat{k}$, let 
$$\textstyle \vec s_t(i)=\sum_{j=\alpha_t(i-1)+1}^{\alpha_t(i)}\vec a_{t,j}$$
and let $\vec S_t=\big(\vec s_t(1),\vec s_t(2),\ldots,\vec s_t(m)\big)$.  For each
$p\in P$ and $t\in\nhat{k}$, let
$$\textstyle d_{t,p}=\tau_p(g_p(x_{t,p,1}),x_{t,p,1})+\frac{1}{p}\tau_p(g_p(x_{t,p,1})-1,x_{t,p,1})\,.$$
Note that for any $j\in\nhat{v_t}$,
$$d_{t,p}=\tau_p(g_p(x_{t,p,j}),x_{t,p,j})+\frac{1}{p}\tau_p(g_p(x_{t,p,j})-1,x_{t,p,j})\,,$$ 
because $\vec x_{t,p}$ is monochromatic with respect to $\ggamma_p$.

Note that, given $i\in\nhat{m}$, $t\in\nhat{k}$, and $\alpha_t(i-1)<j\leq \alpha_t(i)$,
$\textstyle x_{t,p,j}=p^{\mu_p(i)}d_{t,p}+\sum_{l=0}^{\mu_p(i)-2}\tau_p(l,x_{t,p,j})p^l$.
For $p\in P$ and $i\in\nhat{m}$, define 
$$\begin{array}{rl}\textstyle \vec{sm}_p(i)=\sum_{t=1}^k\left(\right.
&\textstyle\sum_{j=\alpha_t(i-1)+1}^{\alpha_t(i)}\vec a_{t,p}
\sum_{l=0}^{\mu_p(i)-2}\tau_p(l,x_{t,p,j})p^{l-\mu_p(i)}+\\
&\textstyle\sum_{j=\alpha_t(i)+1}^{v_t}\vec a_{t,p}x_{t,p,j}p^{-\mu_p(i)}\left.\right)\end{array}\,.$$

Note that if $M=\max\big\{||\vec a_{t,j}||:t\in\nhat{k}\hbox{ and }j\in\nhat{v_t}\big\}$,
then for each $i
\in\nhat{m}$, $||\vec{sm}_p(i)||\leq\frac{M}{p}\sum_{t=1}^kv_t$ because
$g_p(x_{t,p,j})\leq\mu_p(i)-2$ if $j>\alpha_t(i)$.

For the next three paragraphs, let $p\in P$ be fixed.
We have that $$\textstyle \sum_{t=1}^k\sum_{j=1}^{v_t}x_{t,p,j}\vec a_{t,j}=\vec 0\,.$$
Thus dividing by $p^{\mu_p(1)}$ we have 
$\textstyle\sum_{t=1}^k d_{t,p}\vec s_t(1) +\vec {sm}_p(1)=\vec 0$.

Now let $i\in\nhat{m}$.  Dividing by $\mu_p(i)$, we have
$$\textstyle \sum_{t=1}^k\sum_{j=1}^{\alpha_t(i-1)}\vec a_{t,j}x_{t,p,j}p^{-\mu_p(i)}
+\sum_{t=1}^k d_{t,p}\vec s_t(i) +\vec {sm}_p(i)=\vec 0\,.$$

Thus $-\vec s_1(1)-\frac{1}{d_{1,p}}\vec{sm}_p(1)=\sum_{t=2}^k\frac{d_{t,p}}{d_{1,p}}\vec s_t(1)$
and for $i\in\{2,3,\ldots,m\}$,  
$$\textstyle -\vec s_1(i)-\frac{1}{d_{1,p}}\vec{sm}_p(i)=
\sum_{t=1}^k\sum_{j=1}^{\alpha_t(i-1)}\frac{x_{t,p,j}}{d_{1,p}}p^{-\mu_p(i)}\vec a_{t,j}
+\sum_{t=2}^k\frac{d_{t,p}}{d_{1,p}}\vec s_t(i)\,.$$

For $i\in\{2,3,\ldots,m\}$, let 
$$\begin{array}{rl}C_i=\{\vec w\in\bigtimes_{\delta=1}^m\ber^u:
&\vec w_i\in\big\{\vec a_{t,j}:t\in\nhat{k}\hbox{ and }\\
&\hskip 58 pt j\in\nhat{\alpha_t(i-1)}\big\}\\
&\hbox{and if }\delta\in\nhat{m}\setminus\{i\}\hbox{, then }\vec w_\delta=\vec 0\}\,.\end{array}$$

Let $K$ be the positive cone of $\bigcup_{i=2}^mC_i$, that is, all linear
combinations of members of $\bigcup_{i=2}^mC_i$ with non-negative real coefficients.
Notice that for $(\vec w_1,\vec w_2,\ldots,\vec w_m)\in\bigtimes_{\delta=1}^m\ber^u$,
$(\vec w_1,\vec w_2,\ldots,\vec w_m)\in K$ if and only if $\vec w_1=\vec 0$ and
for each $i\in\{2,3,\ldots,m\}$, $\vec w_i$ is a linear combination with non-negative coefficients of
$\{\vec a_{t,j}:(t,j)\in\bigcup_{l=1}^{i-1}J(l)\}$.

Let $L$ be the positive cone of $\{\vec S_2,\vec S_3,\ldots,\vec S_k\}\cup
\bigcup_{i=2}^mC_i$.  We then have that for each $p\in P$,
$-\vec S_1-\frac{1}{d_{1,p}}\big(\vec{sm}_p(1),\vec{sm}_p(2),\ldots,\vec{sm}_p(m)\big)\in L$.
Now $L$ is closed in $\bigtimes_{\delta=1}^m\ber^u$ and for each 
$p\in P$, $\big ||\big(\vec{sm}_p(1),\vec{sm}_p(2),\ldots,\vec{sm}_p(m)\big)\big||
\leq \frac{mM}{p}\sum_{t=1}^kv_t$.  
Therefore $-\vec S_1\in L$.  And since all entries of all 
of the vectors generating $L$ are rational, in fact
$-\vec S_1$ is a linear combination of members of
$\{\vec S_2,\vec S_3,\ldots,\vec S_k\}\cup
\bigcup_{i=2}^mC_i$ with all coefficients non-negative rational
numbers. (See, for example, \cite[Lemma 15.23]{HS}.) Thus there exist
non-negative rational numbers $b_{1,2},b_{1,3},\ldots,b_{1,k}$
such that $-\vec S_1-\sum_{t=2}^kb_{1,t}\vec S_t\in K$.  Letting 
$b_{1,1}=1$, we have $-\sum_{t=1}^kb_{1,t}\vec S_t\in K$.

Similarly, for each $r\in\{2,3,\ldots,k\}$ there exist non-negative
rationals $b_{r,1},\break
b_{r,2},\ldots,b_{r,k}$ with $b_{r,r}=1$ 
such that $-\sum_{t=1}^kb_{r,t}\vec S_t\in K$.  

Thus we have $-\sum_{r=1}^k\sum_{t=1}^kb_{r,t}\vec S_t\in K$ so
$-\sum_{t=1}^k\sum_{r=1}^kb_{r,t}\vec S_t\in K$.
Since each $b_{r,t}\geq 0$ and $b_{r,r}=1$, we have for each
$t$ that $\sum_{r=1}^kb_{r,t}\geq 1$.
For $t\in\{2,3,\ldots,k\}$, let
$$c_t=\frac{\textstyle\sum_{r=1}^k b_{r,t}}{\textstyle\sum_{r=1}^k b_{r,1}}\,.$$
Then $-\vec S_1-\sum_{t=2}^kc_t\vec S_t\in K$ and in fact is
a linear combination with non-negative rational coefficients of
members of $\bigcup_{i=2}^mC_i$.  Recalling the description
of what it means to be in $K$, we see that 
$(\begin{array}{ccccc}A_1&c_2A_2&c_3A_3&\ldots&c_kA_k\end{array})$ satisfies
the columns condition with column partition 
$\{J(1),J(2),\ldots,J(m)\}$. \end{proof}

Notice the amusing fact that in each case, the proof establishes that the 
columns condition is satisfied with the sum of each set
of columns a linear combination of the previous columns using no positive coefficients at all.

\section{Some corollaries}

An immediate corollary of Theorem \ref{thmdkpr} is the following
computable characterisation of doubly IPR.
\begin{corollary}\label{cordiprone} Let $u,v\in\ben$ and let $A$ be a 
$u\times v$ matrix with entries from $\beq$.  Then $A$ is doubly IPR if and only
if there exists $b\in\beq^{+}$ such that the matrix
$(\begin{array}{cc}A&-bI_u\end{array})$ is KPR.\end{corollary}

\begin{proof} We know that $A$ is doubly IPR if and only if the
pair $(A, -I_u)$ is doubly KPR, so Theorem \ref{thmdkpr} applies. \end{proof}

One of the characterisations of image partition regularity is the
following.

\begin{theorem}\label{thmiprchar} Let $u,v\in\ben$ and let $A$ be a 
$u\times v$ matrix with entries from $\beq$.  Then $A$ is  IPR if and only
if there exist $b_1,b_2,\ldots ,b_v\in\beq^{+}$ such that the matrix
$$\left(\begin{array}{c} \begin{array}{ccccc}
b_1&0&0&\ldots&0\\
0&b_2&0&\ldots&0\\
0&0&b_3&\ldots&0\\
\vdots&\vdots&\vdots&\ddots&\vdots\\
0&0&0&\cdots&b_v\end{array}\\
\ \\
A\end{array}\right)$$
is image partition regular.
\end{theorem}

\begin{proof} \cite[Theorem 2.10]{HLS}.\end{proof}

We show now, as a corollary to Theorem \ref{thmdkpr}, that
$A$ is doubly IPR if and only if one can choose $b_1=b_2=\ldots=b_v$
in Theorem \ref{thmiprchar}.

\begin{corollary}\label{cordiprtwo} Let $u,v\in\ben$ and let $A$ be a 
$u\times v$ matrix with entries from $\beq$.  Then $A$ is doubly IPR if and only
if there exists $b\in\beq^{+}$ such that the matrix
$$\left(\begin{array}{c}
bI_v \\
A\end{array}\right)$$
is image partition regular.\end{corollary}

\begin{proof} Using Corollary \ref{cordiprone}, we show that for $b\in\beq^{+}$,
$(\begin{array}{cc}A&-bI_u\end{array})$ is KPR if and only if 
$\left(\begin{array}{c}bI_v\\ A\end{array}\right)$ is IPR.  So let $b\in\beq^+$ be given.

For sufficiency, let $\ben$ be finitely coloured and pick $\vec x\in\ben^v$ such that
$\vec z=\left(\begin{array}{c}bI_v\\ A\end{array}\right)\vec x$ is monochromatic.
Then $\vec z=\left(\begin{array}{c}b\vec x\\ A\vec x\end{array}\right)$
so $(\begin{array}{cc}A&-bI_u\end{array})\vec z=bA\vec x-bA\vec x=\vec 0$.

For necessity, pick $m,n\in\ben$ such that $b=\frac{m}{n}$.
By Lemma \ref{lemnN}, pick monochromatic $\vec z\in (m\ben)^{v+u}$
such that $(\begin{array}{cc}A&-bI_u\end{array})\vec z=\vec 0$.  Pick $\vec w\in(m\ben)^v$
and $\vec y\in(m\ben)^u$ such that $\vec z=\left(\begin{array}{c}\vec w\\ \vec y\end{array}\right)$
and let $\vec x=\frac{1}{b}\vec w$.  Since the entries of $\vec w$ are multiples of $m$, 
$\vec x\in\ben^v$. Since $A\vec w-b\vec y=\vec 0$ we have $b\vec y=A\vec w=bA\vec x$ so
$\vec y=A\vec x$. Therefore $$\left(\begin{array}{c}bI_v\\ A\end{array}\right)\vec x=
\left(\begin{array}{c}b\vec x\\ A\vec x\end{array}\right)=
\left(\begin{array}{c}\vec w\\ \vec y\end{array}\right)=\vec z\,.$$
\end{proof}

 We briefly consider the following question. If the entries of $A$ are integers and $A$ is doubly IPR, must there exist a 
positive integer $b$ such that $(\begin{array}{cc}A&-bI_u\end{array})$
is KPR? In fact, this need not be the case.

\begin{theorem}\label{notint} There is a $2\times 3$ matrix $A$ which is doubly IPR but such that there does
not exist a positive integer $b$ such that $(\begin{array}{cc}A&-bI_u\end{array})$
is KPR.
\end{theorem}

\begin{proof} Let $A=\left(\begin{array}{ccc}4&-4&2\\
5&-5&3\end{array}\right)$.  Then the matrix $\left(\begin{array}{ccccc}4&-4&2&-\frac{1}{2}&0\\
5&-5&3&0&-\frac{1}{2}\end{array}\right)$ satisfies the columns condition (with
$I_1=\{1,2\}$, $I_2=\{3,5\}$ and $I_3=\{4\}$) so by Corollary \ref{cordiprone},
$A$ is doubly IPR.  

The only value of $b$ other than $b=\frac{1}{2}$ for which
$\left(\begin{array}{ccccc}4&-4&2&-b&0\\
5&-5&3&0&-b\end{array}\right)$ satisfies the columns condition 
is $b=-2$.  \end{proof}

However, if we demand that no nonempty set of columns of $A$ sums to $\vec 0$,
we do get the desired result.

\begin{corollary}\label{intb} Let $u,v\in\ben$ and let $A$ be a doubly IPR  $u\times v$ matrix
with entries from $\bez$.  If no nonempty set of columns of $A$ sum to $\vec 0$,
then there exists a positive integer $b$ such that $(\begin{array}{cc}A&-bI_u\end{array})$
is KPR.\end{corollary}

\begin{proof} Pick by Corollary \ref{cordiprone} a positive rational $b$ such that
$(\begin{array}{cc}A&-bI_u\end{array})$ is KPR and pick $m$ and $I_1$, $I_2$,\ldots,
$I_m$ as guaranteed by the columns condition. Now $I_1$ is not contained in
$\nhat{v}$ so pick $t\in\nhat{u}$ such that $v+t\in I_1$.  Then
$b=\sum_{\{ 1,\ldots,v \} \cap I_1} a_{t,j}$ and is therefore an integer.\end{proof}

We conclude by relating the property of being multiply KPR to central subsets of $\ben$. If $S$ is any
discrete space,  its Stone-\v{C}ech compactification $\beta S$ can be regarded as the set of ultrafilters on $S$, with the topology defined by choosing the sets of the form 
$\overline A=\{p\in\beta S:A\in p\}$, where $A$ denotes a subset of $S$, as a base for the open sets. The semigroup operation of $S$ can be extended to $\beta S$ in such a way that $\beta S$ becomes a compact right topological
semigroup, with the property that, for every $s\in S$ the mapping $x\mapsto sx$ from $\beta S$ to itself, is continuous. Any compact right topological semigroup has a smallest ideal which contains an idempotent. An idempotent of this kind is called
 {\it minimal\/},  and a subset of $S$ which is a member of a minimal  idempotent   is called   {\it central\/}.  These  sets have very rich combinatorial properties. The reader is referred to \cite{HS} for further information.

We regard $\beta\ben$ as a semigroup, with the semigroup operation $+$ being the extension of  addition  
on $\ben$. We also regard $\ben$ as embedded in $\beq$ and $\beta\ben$ as embedded in $\beta\beq_d$, where
$\beq_d$ is the set $\beq$ with the discrete topology.
Hence, if $c\in \beq$ and $p\in \beta\ben$, $cp\in \beta\beq_d$ is defined by the 
fact that the operation of multiplication on $\beq$ extends to $\beta\beq_d$, and the map $p\mapsto cp$
is continuous.

\begin{definition}\label{firstentries} A finite matrix over $\beq$ is said to be a {\it first entries\/}
matrix if no row is identically zero, the first non-zero entry of each row is positive and the first non-zero
entries of two different rows are equal if they occur in the same column. A first entries matrix is said to be {\it unital\/} if the first non-zero entry of each row is 1. \end{definition}

\begin{theorem}\label{central} Let $u,k,v_1,v_2,\ldots,v_k\in\ben$ with $k\geq 2$. Let $p$ be a minimal idempotent in $\beta\ben$. For each
$t\in\{1,2,\ldots,k\}$, let $A_t$ be a $u\times v_t$ matrix with entries from $\beq$. Then
$(A_1\:A_2\:\ldots\:A_k)$ is multiply KPR if and only if there exist minimal idempotents
$p_1,p_2,\ldots,p_k$ in $\beta\ben$, with $p=p_1$, with the following property: given members $C_1,C_2,\ldots,C_k$ of $p_1,p_2,\ldots,p_k$ respectively, there exists $\vec x_t\in C^{v_t}$ for each
$t\in\{1,2,\ldots,k\}$ such that $c_t\vec x_t\subseteq C_t^{v_t}$ and
 $\sum_{t=1}^k\,A_tc_t\vec x_t=\vec O$. \end{theorem}

\begin{proof} The condition stated is obviously sufficient for $(A_1\:A_2\:\ldots\:A_k)$ to be
multiply KPR because, given any finite colouring of $\ben$, every element of $\beta\ben$
has a member which is monochrome.

To prove that it is necessary, assume that $(A_1\:A_2\:\ldots\:A_k)$ is multiply KPR.  By
Theorem \ref{thmdkpr}, there exist $c_1,c_2,c_3,\ldots,c_k\in\beq^+$, with $c_1=1$,
such that $A=(c_1A_1\:c_2A_2\:c_3A_3\:\ldots\:c_kA_k)$ is KPR. 
 For each $t\in\{1,2,\ldots,k\}$, let $p_t=c_tp$.
It follows from \cite[Lemma 5.19.2]{HS} that $p_t$ is also a minimal idempotent in $\beta\ben$. Let $C_t\in p_t$ for each $t\in\{1,2,\ldots,k\}$.

 Let $v=v_1+v_2+\ldots +v_k$.  Since $A$ satisfies the
columns condition, there exists $m\in \ben$ and a $v\times m$ unital first entries matrix $G$ over $\beq$ such that $AG=O$. We can write $G$ in block form as $G=\left(\matrix{G_1\cr G_2\cr \vdots\cr G_t}\right)$, where, for each $t\in\{1,2,\cdots,k\}$, $G_t$ is a $v_t\times m$
 matrix over $\beq$. Let $C=\bigcap_{t=1}^k\, c_t^{-1}C_t$. Since
$C\in p_1$, $C$ is a central subset of $\ben$. By \cite[Lemma 2.8]{HLS}, there exists
$\vec x\in \ben^m$ such that all the entries of $G\vec x$ are in $C$. Put $\vec x_t=G_t \vec x$ for each $t\in\{1,2,\ldots,k\}$. Then all the entries of $\vec x_t$ are in $C$, and so all the entries of $c_t\vec x_t$ are in $C_t$.  Furthermore, $\sum_{t=1}^k\,A_tc_t\vec x_t=\sum_{t=1}^k\,c_tA_tG_t\vec x=AG\vec x=\vec 0$. 
\end{proof}

In a similar vein, the following characterisation of doubly IPR matrices follows
very easily from Corollary \ref{cordiprtwo}. Let $u,v\in\ben$ and let $A$ be a $u\times v$
matrix over $\beq$. Then $A$ is doubly IPR if and only if, for every minimal idempotent $p\in\beta\ben$, there exists a minimal idempotent $q\in\beta\ben$ such that, whenever $B\in p$
and $C\in q$, there exists $\vec x\in B^v$ satisfying $A\vec x\in C^u$.

$$\begin{array}{ll}\hbox{Dennis Davenport}&\hbox{Neil Hindman}\\
\hbox{Department of Mathematics}&\hbox{Department of Mathematics}\\
\hbox{Howard University}&\hbox{Howard University}\\
\hbox{Washington, DC 20059 USA}&\hbox{Washington, DC 20059 USA}\\
\hbox{\tt dennis.davenport@live.com}&\hbox{\tt nhindman@aol.com}\\
 \\
\hbox{Imre Leader}&\hbox{Dona Strauss}\\
\hbox{Department of Pure Mathematics}&\hbox{Department of Pure Mathematics}\\
\hbox{\hskip 15 pt and Mathematical Statistics}&\hbox{University of Leeds}\\
\hbox{Centre for Mathematical Sciences}&\hbox{Leeds LS2 9J2, UK}\\
\hbox{Wilberforce Road}&\hbox{\tt d.strauss@hull.ac.uk}\\
\hbox{Cambridge, CB3 0WB, UK}\\
\hbox{\tt i.leader@dpmms.cam.ac.uk}\end{array}$$

\end{document}